\documentclass[11pt]{article} 

\usepackage[utf8]{inputenc} 

\usepackage{geometry} 
\usepackage{graphicx} 
\usepackage{mathtools}
\usepackage{comment}
\usepackage{amssymb} 
\usepackage{amsthm} 
\usepackage{color}
\usepackage{mathrsfs}
\usepackage{url} 
\usepackage{float}
\usepackage{amsmath}
\usepackage{tikz}
\theoremstyle{definition}
\newtheorem{mythm}{Theorem}[section]

\newtheorem{mycor}{Corollary}[section]
\newtheorem{mylem}{Lemma}[section]

\newtheorem{myexa}{Example}[section]
\newtheorem{myrem}{Remark} [section]

\usepackage{booktabs} 
\usepackage{array} 
\usepackage{paralist} 
\usepackage{verbatim} 
\usepackage{subfig} 

\usepackage{fancyhdr} 
\usepackage{sectsty}
\allsectionsfont{\sffamily\mdseries\upshape} 

\usepackage[nottoc,notlof,notlot]{tocbibind} 
\usepackage[titles,subfigure]{tocloft} 

\title{Minimal Number of Steps in Euclidean Algorithm and its Application to Rational Tangles}
\author{M. Syafiq Johar}
\date{}

\begin{document}
\maketitle
\vspace{4mm}
\begin{abstract}
We define the regular Euclidean algorithm and the general form which leads to the method of least absolute remainders and also the method of negative remainders.  We are going to show that if looked from the perspective of subtraction, the method of least absolute remainders and the regular method have the same number of steps which is in fact the minimal number of steps possible. This enables us to apply the theory in rational tangles to determine the most efficient way to untangle a rational tangle.
\end{abstract}

\vspace{4mm}

\section{Introduction}
The Euclidean algorithm is a method of finding the greatest common divisor (GCD) of two numbers by repeatedly doing the division algorithm on pair of numbers. The method goes way back in ancient Greek, first described by Euclid in his mathematical treatise, \emph{Elements}. We will see the original algorithm, along with some special variants in Sections \ref{section2} and \ref{section3} of this paper. 

Despite its simplicity, the Euclidean algorithm is one of the fundamental topics in any number theory lectures and many deeper results have been built up upon it, ranging from algebra to cryptography and, as we will see in this paper, knot theory.

The connection between the purely number theoretic Euclidean algorithm and the more topological knot theory is clear in the topic of rational tangles. Rational tangles are first described by John Conway. It is a special type of knot (or more precisely, a tangle) which is constructed from two strands of strings using only two moves: twisting and rotation. 

In some literature, rotation is often replaced by reflection. However, in this paper, we will consider rotation since it is more physically feasible for practical purposes. We will see in Section \ref{section6} on the construction of rational tangles and how the Euclidean algorithm is applied in the process of untangling them. This connection, along with continued fractions, has been recognised in some knot theory books (for example, in Cromwell's book on knots and links \cite{cromwell}). However, in this paper, we will take it a step further by proving a special property of the Euclidean algorithm in Section \ref{section5} which will help us optimise the untangling algorithm.

\section{The Euclidean Algorithm} \label{section2}
We begin by defining the Euclidean algorithm. If we were given two positive integers $x_0$ and $x_1$ such that $x_0 > x_1$, by the division algorithm, we can write $x_0$ as:
$$x_0=x_1q_1+x_2$$
where $q_1$, is the greatest integer smaller than or equal to $\frac{x_0}{x_1}$ and $x_2$ is the remainder of the division of $\frac{x_0}{x_1}$ such that $0 \leq x_2 <x_1$. We can also say that we obtain $x_2$ by successively subtracting $x_1$ from $x_0$ until we get a positive integer strictly less than $x_1$. Using $x_1$ and $x_2$, we repeat the division algorithm to get the quotient $q_2$ and the remainder $x_3$ and so on until we get no more remainders. This can be written as:
\begin{eqnarray} \label{euclid}
x_0&=& x_1q_1+x_2 \nonumber
\\ x_1&=& x_2q_2+x_3 \nonumber
\\  & \vdots
\\ x_{n-2}&=& x_{n-1}q_{n-1}+x_{n} \nonumber
\\ x_{n-1}&=& x_nq_n. \nonumber
\end{eqnarray}

From this algorithm, it can be shown that the GCD of $x_0$ and $x_1$ is indeed $x_n$. For the proof, see the book by Burton \cite[p.31]{burton}. In this algorithm, we are required to do $n$ divisions to get the greatest common divisor (GCD) of $x_0$ and $x_1$, hence there are a total of $n$ equations in (\ref{euclid}).

\section{The General Euclidean Algorithm and the Method of Least Absolute Remainders} \label{section3}
An alternative method to find the GCD of two numbers is the method of least absolute remainders. If we allow negative remainders, we can choose a different remainder than the ones in (\ref{euclid}). For example, we look at the first equation. Rather than choosing $x_2$ as the remainder, we can take away another $x_1$ from $x_0$ to get a negative remainder of $(x_2-x_1)$ as follows:
\begin{eqnarray*}
x_0 &=&x_1q_1+x_2
\\ &=& x_1(q_1+1)+(x_2-x_1).
\end{eqnarray*}
This way, we can have a freedom of choosing which remainder to work with in the next step (for negative remainders, we work with its absolute value in the following step). Thus, a general way of writing the Euclidean algorithm is given by:
\begin{eqnarray} \label{euclidany}
x_0&=& x_1p_1+\epsilon_2 x_2 \nonumber
\\ x_1&=& x_2p_2+\epsilon_3 x_3 \nonumber
\\  & \vdots
\\ x_{m-2}&=& x_{m-1}p_{m-1}+\epsilon_{m}x_{m} \nonumber
\\ x_{m-1}&=& x_mp_m \nonumber
\end{eqnarray}
where $x_i$ are all positive integers such that $x_i<x_{i+1}$ and $\epsilon_i=\pm1$, depending on which remainder is chosen \cite[p.156]{goodman}. If we choose all $\epsilon_i$ to be $+1$, we would get the regular Euclidean algorithm as in the previous section. Note that there are $m$ equations in (\ref{euclidany}), hence there are $m$ divisions involved in the algorithm. 

\begin{myexa} \label{example0}
If we choose $x_0=3$ and $x_1=2$, we have two possible algorithms.

\begin{minipage}{0.5\textwidth}
\begin{eqnarray*} 
3&=& 2(1)+1
\\ 2&=& 1(2)+0
\end{eqnarray*}
\end{minipage} \hspace{-10 mm}
\begin{minipage}{0.5\textwidth}
\begin{eqnarray*} 
3&=& 2(2)-1
\\ 2&=& 1(2)+0.
\end{eqnarray*}
\end{minipage}

\vspace{5mm}
\noindent If we choose $x_0=4$ and $x_1=3$, we have three possibilities:

\begin{minipage}{0.3\textwidth}
\begin{eqnarray*} 
4&=& 3(1)+1
\\ 3&=& 1(3)+0
\\
\end{eqnarray*}
\end{minipage} 
\begin{minipage}{0.3\textwidth}
\begin{eqnarray*} 
4&=& 3(2)-2
\\ 3&=& 2(1)+1
\\ 2&=&1(2)+0
\end{eqnarray*}
\end{minipage} 
\begin{minipage}{0.3\textwidth}
\begin{eqnarray*} 
4&=& 3(2)-2
\\ 3&=& 2(2)-1
\\ 2&=&1(2)+0.
\end{eqnarray*}
\end{minipage}
\end{myexa}
\vspace{3mm}
The method of least absolute remainders is such that we always choose the remainder with the smaller absolute value in each step. In other words, we always choose the remainder $x_i$ such that $2x_i \leq x_{i-1}$ for $i=2,3,\ldots,m$.  So, in Example \ref{example0} above, for $x_0=3$ and $x_1=2$, both calculations are carried out using the method of least absolute remainders and for $x_0=4$ and $x_1=3$, only the first algorithm is carried out using the method of least absolute remainders. In the latter case, the method of least absolute remainders coincides with the regular Euclidean algorithm.

Note that the third method for the case $x_0=4$ and $x_1=3$ gives negative remainders for all the equations in the algorithm. When we require all $\epsilon_1$ to be $-1$ in (\ref{euclidany}), we call this method the method of negative remainders. We will not elaborate on this method but it will be mentioned again later when we apply the idea of Euclidean algorithms to rational tangles.

\begin{myexa} \label{example1}
Let us compare the regular euclidean algorithm with the method of least absolute remainders. Consider both versions of the Euclidean algorithms  for the numbers $x_0=807$ and $x_1=673$.

\begin{minipage}{0.5\textwidth}
\begin{eqnarray*} 
807&=& 673(1)+134
\\ 673&=& 134(5)+3
\\ 134&=& 3(44)+2
\\ 3&=& 2(1)+1
\\ 2&=& 1(2)+0
\end{eqnarray*}
\end{minipage} \hspace{-10 mm}
\begin{minipage}{0.5\textwidth}
\begin{eqnarray*} 
807&=& 673(1)+134
\\ 673&=& 134(5)+3
\\ 134&=& 3(45)-1
\\ 3&=& 1(3)+0.
\\
\end{eqnarray*}
\end{minipage}
\end{myexa}
\vspace{3 mm}
The number of divisions required in the method of least absolute remainders is smaller than the regular version. This is generally true for any pair of numbers by a theorem proven by Leopold Kronecker stating that the number of divisions in the method of least absolute remainders is not more than any other Euclidean algorithm \cite[p.47-51]{uspensky}. 

Also, in 1952, A.W. Goodman and W.M. Zaring from the University of Kentucky proved that the number of divisions in the regular Euclidean algorithm and the method of least absolute remainders differ by the number of negative remainders obtained in the method of least absolute remainders \cite[p.157]{goodman}. In mathematical terms:
$$ n-m=\frac{1}{2}\sum^m_{i=2}(|\epsilon_i|-\epsilon_i).$$

From Example \ref{example1}, in the regular Euclidean algorithm, there are $5$ divisions involved. However, in the method of least absolute remainders, there are $4$ divisions involved. Thus, there is one negative remainder somewhere in the algorithm for the method of least absolute remainders, which can be seen in the example.

\section{Number of Steps for Euclidean Algorithm} \label{section4}
What if we consider the Euclidean algorithm using subtraction rather than division i.e. we consider taking away $x_1$ from $x_0$ as one step and moving on from working with the pair $x_0$ and $x_1$ to the pair $x_1$ and $x_2$ as one step. For example, consider again Example \ref{example1}. We begin with $807$. Taking away $673$ from it leaves us with $134$ which is smaller than $673$, so we can't take anymore $673$ from it lest it will become negative. We then move on from this to work with the pair $673$ and $134$. We can list out the method by rearrangement of the Euclidean algorithm in such a way:
\begin{eqnarray*} 
807- 673(1)&=& 134
\\ 673-134(5)&=& 3
\\ 134- 3(44)&=& 2
\\ 3-2(1)&=& 1
\\ 2-1(2)&=& 0.
\end{eqnarray*}

The number of steps in the calculation can be determined by figuring out how many times subtraction is done plus how many times we switch the number pairs to work with. From the example, in the calculation above, we carried out a total of $1+5+44+1+2=53$ subtractions and $4$ swaps. Therefore, we have a total of $57$ steps.

Interestingly, if we consider the method of least absolute remainders, which can be obtained via a similar rearrangement, but this time allowing negative residues, we have:
\begin{eqnarray*} 
807-673(1)&=& 134
\\ 673-134(5)&=& 3
\\ 134-3(45)&=& -1
\\ 3-1(3)&=&0.
\end{eqnarray*}

In this algorithm, there are $1+5+45+3=54$ subtractions and $3$ swaps, making it a total of $57$ steps as well. How are the two algorithms related? It can be shown that the number of steps for these two algorithms are the same. The proof of the following theorem is based on a proof by Goodman and Zaring \cite{goodman}.

\begin{mythm} \label{samenum}
The number of steps required to carry out the Euclidean algorithm using the regular method and the method of least absolute remainders are the same.
\end{mythm}

\begin{proof}
Suppose that we have the set of $n$ equations in the regular Euclidean algorithm as in (\ref{euclid}). The number of swaps are $n-1$ and thus, the number of steps required to reduce it down to $0$ is $S=(n-1)+ \sum^n_{j=1}q_j$. If $2x_i \leq x_{i-1}$ for all $i=2,3,\ldots,n$, we are done as this regular Euclidean algorithm is identically the same as the method of least absolute remainders. 

Suppose that there exists some $i=2,3,\ldots,n-1$ such that $\frac{x_{i-1}}{2} < x_i<x_{i-1}$. Note that $x_n$ divides $x_{n-1}$ exactly, so $i \neq n$. Choose the smallest such $i$ and consider the three equations: 
\begin{eqnarray} \label{euclidi}
x_{i-2}&=& x_{i-1}q_{i-1}+x_i \nonumber
\\ x_{i-1}&=& x_iq_i+x_{i+1} 
\\ x_i&=& x_{i+1}q_{i+1}+x_{i+2}. \nonumber
\end{eqnarray}
We want to transform this set of equations to the method of least absolute remainders so we want to get rid of the remainder $x_i$ as it is bigger than $\frac{x_{i-1}}{2}$. In order to do so, we manipulate the first equation in (\ref{euclidi}) to get: 
\begin{equation} \label{manipulated} x_{i-2}= x_{i-1}(q_{i-1}+1)-(x_{i-1}-x_i).\end{equation}
Also, by our assumption, $\frac{x_{i-1}}{2} < x_i<x_{i-1} \Rightarrow 1<\frac{x_{i-1}}{x_i}<2$ and hence $q_i=1$. Thus the second equation in (\ref{euclidi}) becomes:
$$x_{i-1}= x_i+x_{i+1} \Rightarrow x_{i+1}= x_{i-1}-x_i.$$
Putting this in the third equation of (\ref{euclidi}) and equation (\ref{manipulated}), we would get:
\begin{eqnarray*} 
x_{i-2}&=& x_{i-1}(q_{i-1}+1)-x_{i+1} 
\\ x_{i-1}&=& x_{i+1}(q_{i+1}+1)+x_{i+2}. 
\end{eqnarray*}
Therefore, there are $n-1$ equations left in (\ref{euclid}) now. Thus, the number of steps required to reduce it down to $0$ after this manipulation is given by: 
\begin{align*}&((n-1)-1)+\sum^{i-2}_{j=1}q_j+(q_{i-1}+1)+(q_{i+1}+1)+\sum^n_{j=i+2}q_j
\\= \ & (n-1)+\sum^n_{\substack{j=1 \\ j\neq i}}q_j+1\ = \ (n-1)+\sum^n_{\substack{j=1 \\ j\neq i}}q_j+q_i \ = \ (n-1)+\sum^n_{j=1}q_j.
\end{align*}
This is the same as the number of steps before the manipulation. By induction, if we continue down the algorithm and get rid of all the remainders $x_i$ that are greater than $\frac{x_{i-1}}{2}$, we get a constant number of steps for getting it down to $0$.
\end{proof}

\section{Minimum Number of Steps} \label{section5}
In this section, we are going to prove that indeed, the number of steps that is obtained from the method of least absolute remainders is the minimum possible among any other general Euclidean algorithm. 

We are going to use Kronecker's proof \cite[p.47-51]{uspensky} as a loose base for our proof to show that the method of least absolute remainders requires the least number of steps among any other Euclidean algorithm. The main difference between Kronecker's proof and our proof here is the way we defined the number of steps i.e. in our case, we have taken the value of the remainders into account as we have defined in the previous section. However, the general idea of proving it via induction and splitting it into cases are similar.

We begin with two lemmas. We denote the number of steps required to find the GCD of two numbers $a$ and $b$ using the method of least absolute remainders as $S_m(a,b)$ and the number of steps for any Euclidean algorithm as $S(a,b)$.

\begin{mylem} \label{easylemma}
If $x_1$ divides $x_0$ exactly i.e. $kx_1=x_0$ for some positive integer $k$, the number of steps for finding the GCD of $x_0$ and $x_1$ are the same for any method. Precisely, $S(x_0,x_1)=S_m(x_0,x_1)=k$.
\end{mylem}

\begin{mylem} \label{lemma}
If $2x_1<x_0$, the number of steps for finding the GCD of $x_0$ and $x_1$ using the method of least absolute remainders is exactly one less than the number of steps for finding the GCD of $x_0$ and $x_0-x_1$ using the same method. In other words, $S_m(x_0,x_1) +1= S_m(x_0,x_0-x_1)$.
\end{mylem}

\begin{proof}
We are going to prove this by strong induction. We begin with the case $x_0=3$. The only possible value for $x_1$ is $1$.

\begin{minipage}{0.5\textwidth}
\begin{eqnarray*} 
3&=&1(3)+0
\\ S_m(3,1)&=&3
\\
\end{eqnarray*}
\end{minipage} \hspace{-15 mm}
\begin{minipage}{0.5\textwidth}
\begin{eqnarray*} 
3&=& 2(1)+1
\\ 2&=&1(2)+0
\\ S_m(3,2)&=&1+1+2=4.
\end{eqnarray*}
\end{minipage}
\vspace{5 mm}

\noindent Similarly, for the case $x_0=4$, we have $x_1=1$ as the only possibility.

\begin{minipage}{0.5\textwidth}
\begin{eqnarray*} 
4&=&1(4)+0
\\ S_m(4,1)&=&4
\\
\end{eqnarray*}
\end{minipage} \hspace{-15 mm}
\begin{minipage}{0.5\textwidth}
\begin{eqnarray*} 
4&=& 3(1)+1
\\ 3&=&1(3)+0
\\ S_m(4,3)&=&1+1+3=5.
\end{eqnarray*}
\end{minipage}
\vspace{4 mm}

\noindent However, for the case $x_0=5$, we have two possibilities for $x_1$ which are $1$ and $2$.

\begin{minipage}{0.5\textwidth}
\begin{eqnarray*} 
5&=&1(5)+0
\\ S_m(5,1)&=&5
\\
\end{eqnarray*}
\end{minipage} \hspace{-15 mm}
\begin{minipage}{0.5\textwidth}
\begin{eqnarray*} 
5&=& 4(1)+1
\\ 4&=&1(4)+0
\\ S_m(5,4)&=&1+1+4=6
\end{eqnarray*}
\end{minipage}

\begin{minipage}{0.5\textwidth}
\begin{eqnarray*} 
5&=&2(2)+1
\\2&=&1(2)+0
\\ S_m(5,2)&=&2+1+2=5
\\
\end{eqnarray*}
\end{minipage} \hspace{-7.5 mm}
\begin{minipage}{0.5\textwidth}
\begin{eqnarray*} 
5&=& 3(1)+2
\\ 3&=&2(1)+1
\\2&=&1(2)+0
\\ S_m(5,3)&=&1+1+1+1+2=6.
\end{eqnarray*}
\end{minipage}
\vspace{5 mm}

\noindent So, in all the cases above, we can see that $S_m(x_0,x_1)+1= S_m(x_0,x_0-x_1)$. Now, for induction, we assume that this is true for all $2x_1<x_0<n$. We are going to prove it for the case $2x_1<x_0=n$. We break the problem into three cases:

\begin{description}
\item[First case: $2x_1<x_0\leq\frac{5}{2}x_1$] \hfill
\\ For this case, for the numbers $x_0$ and $x_1$, the method of least absolute remainders begins with $x_0=2x_1+x_2$ and continues with $x_1=x_2q_2+\epsilon_3x_3$. For the other pair $x_0$ and $x_0-x_1$, we have the equation $x_0=2(x_0-x_1)-(x_0-2x_1)=2(x_0-x_1)-x_2$. This is calculated using the method of least absolute remainders as $2x_2=2x_0-4x_1<x_0-x_1$ by the assumption that $x_0\leq 3x_1$. We continue with the next equation relating $x_0-x_1$ and $x_1$. i.e. $x_0-x_1=x_1+(x_0-2x_1)$. This also calculated using the method of least absolute remainders as $x_0\leq\frac{5}{2}x_1 \Rightarrow 2(x_0-2x_1)<x_1$. We compare the two algorithms:
\vspace{-4mm}

\begin{minipage}{0.5\textwidth}
\begin{eqnarray*} 
(1) \ x_0&=&2x_1+x_2
\\(2) \ x_1&=&x_2q_2+\epsilon_3x_3
\\
\\
\end{eqnarray*}
\end{minipage} \hspace{-15 mm}
\begin{minipage}{0.5\textwidth}
\begin{eqnarray*} 
(1) \hspace{9.7mm} x_0&=&2(x_0-x_1)-x_2
\\ (2) \ x_0-x_1&=&x_1+(x_0-2x_1)
\\ &=&x_1+x_2
\\ &=& x_2(q_2+1)+\epsilon_3x_3.
\end{eqnarray*}
\end{minipage}

\vspace{3mm}
If we continue carrying out the algorithm for both pairs using the method of least absolute remainders, the rest of the calculations are the identical because we will carry out the following steps in both calculations using $x_2$ and $x_3$ as the starting pair. Therefore, we can calculate the number of steps for the algorithm:
\begin{eqnarray*}
S_m(x_0,x_1)+1&=&(2+1+q_2+1+S_m(x_2,x_3))+1
\\&=&2+1+(q_2+1)+1+S_m(x_2,x_3)
\\ &=& S_m(x_0,x_0-x_1).
\end{eqnarray*}
\item[Second case: $\frac{5}{2}x_1<x_0\leq3x_1$] \hfill
\\
For this case, the method of least absolute remainders begins with $x_0=3x_1-x_2$. If we continue with the pair $x_1$ and $x_2$ using the method of least absolute remainders, we have: 
$$S_m(x_0,x_1)=3+1+S_m(x_1,x_2)=4+S_m(x_1,x_2).$$

 As in the previous case, for the pair $x_0$ and $x_0-x_1$, we begin with  $x_0=2(x_0-x_1)-(x_0-2x_1)=2(2x_1-x_2)-(x_1-x_2)$. We continue with the pair $2x_1-x_2$ and $x_1-x_2$ to get the equation:  $2x_1-x_2=(x_1-x_2)+x_1$. We wish to relate this with the pair $x_1$ and $x_1-x_2$ so that we can use the inductive hypothesis. 

Note that the equation relating $x_1$ and $x_1-x_2$ is $x_1=(x_1-x_2)q+r$ where $2r<x_1-x_2$. Notice that this is equivalent to $2x_1-x_2=(x_1-x_2)(q+1)+r$ by the relationship $2x_1-x_2=(x_1-x_2)+x_1$. This implies that $S_m(2x_1-x_2,x_1-x_2)=1+S_m(x_1,x_1-x_2)$. Therefore, we have the equation:
\begin{eqnarray*}
S_m(x_0,x_0-x_1)&=&2+1+S_m(x_0-x_1,x_0-2x_1)
\\&=&3+S_m(2x_1-x_2,x_1-x_2)
\\&=&4+S_m(x_1,x_1-x_2).
\end{eqnarray*}
Thus, putting the two results together and using the inductive hypothesis, we have:
\begin{eqnarray*}
S_m(x_0,x_1)+1&=&4+(1+S_m(x_1,x_2))
\\&=&4+S_m(x_1,x_1-x_2)
\\&=&S_m(x_0,x_0-x_1).
\end{eqnarray*}
\item[Third case: $3x_1<x_0$] \hfill
\\
The first step for the method of least absolute remainders for the pair $x_0$ and $x_1$ is $x_0=x_1q_1+\epsilon_2x_2$ such that $2x_2<x_1$. For the pair $x_0$ and $x_0-x_1$, we have $x_0=(x_0-x_1)+x_1$, this is done using the method of least absolute remainders as $3x_1<x_0 \Rightarrow 2x_1<x_0-x_1$. We continue with the pair $x_0-x_1$ and $x_1$. By the relationship $x_0=x_1q_1+\epsilon_2x_2$, we have $x_0-x_1=x_1(q_1-1)+\epsilon_2 x_2$. The rest of the calculations are done using the pair $x_1$ and $x_2$ so they are identical onwards. Therefore, we have:
\begin{eqnarray*}
S_m(x_0,x_1)+1&=&q_1+1+S_m(x_1,x_2)+1
\\ &=&1+1+(q_1-1)+1+S_m(x_1,x_2) 
\\&=& S_m(x_0,x_0-x_1).
\end{eqnarray*}
\end{description}
Therefore, by proving all the three cases, we are done.
\end{proof}

With the two lemmas above, we are going to show that the method of least absolute remainders gives the least number of steps among any other Euclidean algorithm methods. 

\begin{mythm} \label{coolthm}
Let $x_0$ and $x_1$ be positive integers such that $x_1<x_0$. Then, the method of least absolute remainders gives the least number of steps among any other Euclidean algorithm methods for this pair of integers. In other words, $S_m(x_0,x_1) \leq S(x_0,x_1)$.
\end{mythm}

\begin{proof}
Similar to the lemma above, we are going to prove this by induction. The cases for $x_0=3$ and $x_1=2$ have been done in Example \ref{example0}. Note that both methods are considered as the method of least absolute remainders but the first one gives $4$ steps and the second one gives $5$ steps. Both of them are method of least absolute remainders but why the number of steps are different? Here we are going to rule out an ambiguity. 

For the case where $x_0=a(2k+1)$ and $x_1=2a$ for some positive integers $a$ and $k$, we could either choose $a$ or $-a$ as the remainder (i.e. the absolute value of the remainder is exactly half of $x_1$) as both have the same magnitude. In our case, when this happens, we define the method of least absolute remainders by choosing the positive remainder so as to reduce the number of steps in the calculation (if we choose the negative remainder, we have to do one extra step of taking away another $x_1$ from $x_0$). Also, if this case were to happen, the following step will be the last step in the algorithm as in the following step, as implied from Lemma \ref{easylemma}.

Therefore, by this rule, for the pair $x_0=3$ and $x_1=2$ in Example \ref{example0}, the first algorithm is the method of least absolute remainders (and the second one is not) and thus the number of step is minimised using this algorithm.

Next we look at the case $x_0=4$, we have two choices for $x_1$ which are $2$ and $3$. However, for $x_1=2$, there is nothing to be done as $2$ divides $4$, so any algorithm will give the same number of steps by Lemma \ref{easylemma}. We only consider $x_1=3$. This, again, has been done in Example \ref{example0}, and clearly, $S_m(4,3)=5$ which is less than the other two cases listed (7 and 8 respectively).

Now we look at $x_0=5$. There are three choices possible for $x_1$ namely $2,3$ and $4$. We look at all three of them and work out all possible Euclidean algorithms for each pair. The method of least absolute remainders will be the first one in each list.

\hspace{3mm}
\begin{minipage}{0.5\textwidth} 
\begin{eqnarray*} 
5&=& 2(2)+1
\\ 2&=& 1(2)+0
\end{eqnarray*}
\end{minipage} \hspace{-22mm}
\begin{minipage}{0.5\textwidth}
\begin{eqnarray*} 
5&=& 2(3)-1
\\ 2&=& 1(2)+0
\end{eqnarray*}
\end{minipage}

\begin{minipage}{0.3\textwidth}
\begin{eqnarray*} 
5&=& 3(2)-1
\\ 3&=& 1(3)+0
\\ 
\end{eqnarray*}
\end{minipage} 
\begin{minipage}{0.3\textwidth}
\begin{eqnarray*} 
5&=& 3(1)+2
\\ 3&=& 2(1)+1
\\ 2&=&1(2)+0
\end{eqnarray*}
\end{minipage} 
\begin{minipage}{0.3\textwidth}
\begin{eqnarray*} 
5&=& 3(1)+2
\\ 3&=& 2(2)-1
\\ 2&=& 1(2)+0
\end{eqnarray*}
\end{minipage}

\hspace{-8mm}
\begin{minipage}{0.25\textwidth}
\begin{eqnarray*} 
5&=& 4(1)+1
\\ 4&=& 1(4)+0
\\
\\
\end{eqnarray*}
\end{minipage} 
\begin{minipage}{0.25\textwidth}
\begin{eqnarray*} 
5&=& 4(2)-3
\\ 4&=& 3(1)+1
\\ 3&=&1(3)+0
\\
\end{eqnarray*}
\end{minipage} 
\begin{minipage}{0.25\textwidth}
\begin{eqnarray*} 
5&=& 4(2)-3
\\ 4&=& 3(2)-2
\\ 3&=&2(1)+1
\\ 2&=& 1(2)+0
\end{eqnarray*}
\end{minipage}
\begin{minipage}{0.25\textwidth}
\begin{eqnarray*} 
5&=& 4(2)-3
\\ 4&=& 3(2)-2
\\ 3&=&2(2)-1
\\ 2&=& 1(2)+0
\end{eqnarray*}
\end{minipage}

\vspace{5mm}
\noindent We can calculate the number of steps required for all of the three cases and see that the number of steps using the method of least absolute remainders is the least among the other algorithms for the same pair. Note that for the case $x_0=5$ and $x_1=3$, the method of least absolute remainders (the first algorithm) and the regular Euclidean algorithm (the second algorithm) have the same number of steps even though their lengths are different. This is true in any case, according to Theorem \ref{samenum} earlier. 

Now, for induction, we assume that for $x_0<n$ we have $S_m(x_0,x_1)\leq S(x_0,x_1)$. We are going to prove for the case $x_0=n$. Let $x_1<x_0$ which does not divide $x_0$ exactly. Consider the equation $x_0=x_1q_1+\epsilon_2 x_2$ such that $2x_2 \leq x_1$ (obtained using the method of least absolute remainders) and another equation $x_0=x_1q_1'+\epsilon_2'x_2'$ be the first step of any general Euclidean algorithm. We split the problem into two cases:
\begin{description}
\item[First case: $x_2=x_2'$] \hfill
\\ If $\epsilon_2=\epsilon_2'$, then the two equations are the identically the same, thus $q_1=q_1'$, and hence, by the inductive step, since $x_1<n$, we have: 
$$S_m(x_0,x_1)=q_1+1+S_m(x_1,x_2)\leq q_1'+1+S(x_1,x_2)=S(x_0,x_1).$$

If $\epsilon_2=-\epsilon_2'$, by adding the two equations, we would get $2x_0=x_1(q_1+q_1')$. Also, it is clear that $q_1'=q_1+\epsilon_2$. Putting this in either equation will yield $2x_2=x_1$, which forces this situation to be the ambiguous case. This implies that $\epsilon_2=+1$ and $q_1'=q_1+1$. Thus $S_m(x_1,x_2)=S(x_1,x_2)=2$ by Lemma \ref{easylemma} and:
$$S_m(x_0,x_1)=q_1+1+S_m(x_1,x_2)< q_1'+1+S(x_1,x_2)=S(x_0,x_1).$$
\item[Second case: $x_2<x_2'$] \hfill
\\ The method of least absolute remainders begins with $x_0=x_1q_1+\epsilon_2 x_2$ such that $2x_2 \leq x_1$ and any general Euclidean algorithm begins with $x_0=x_1q_1'+\epsilon_2'x_2'$. Therefore, for the case $x_2<x_2'$, it is necessarily true that:
\begin{eqnarray*} 
q_1'&=& q_1+\epsilon_2
\\ \epsilon_2'&=& -\epsilon_2
\\ x_2'&=&x_1-x_2.
\end{eqnarray*}
Hence, for the beginning of the other algorithm, we have the equation:
$$x_0=x_1(q_1+\epsilon_2)+\epsilon_2(x_2-x_1).$$
Thus, applying the inductive hypothesis and Lemma \ref{lemma}, we have:
\begin{eqnarray*}
S(x_0,x_1)&=&(q_1+\epsilon_2)+1+S(x_1,x_1-x_2) 
\\ &\geq& (q_1+\epsilon_2)+1+S_m(x_1,x_1-x_2)
\\ &=& q_1+\epsilon_2+2+S_m(x_1,x_2)
\\ &\geq& q_1+1+S_m(x_1,x_2) 
\\ &=& S_m(x_0,x_1).
\end{eqnarray*}
\end{description}
Therefore, we have proven that the method of least absolute remainders gives the least number of steps among any general Euclidean algorithms.
\end{proof}
From this, we can deduce two of the following corollaries.

\begin{mycor} \label{regularalgmin}
The regular Euclidean algorithm gives the least number of steps among any general Euclidean algorithms.
\end{mycor}

\begin{mycor} \label{regularalgmin2}
The method of least absolute remainders gives the least number of steps with the least number of equations among any general Euclidean algorithm.
\end{mycor}

\begin{proof}
These are immediate from Theorem \ref{samenum}, Theorem \ref{coolthm} and the theorem by Kronecker.
\end{proof}

\section{Application in Rational Tangles}\label{section6}
\subsection{Rational Tangles and Conway's Theorem} 
We begin by defining the untangle as two strings lying vertical and not intersecting. We assign the number $0$ to this setting. Label each end of the strings as NE, NW, SE and SW respectively. We have two operations that we can do on the tangle. Firstly, we can switch the position of the SE and NE ends. This move is called twist. If we twist the strings such that the gradient of the overstrand is positive, we call it the positive twist. Otherwise, we call it the negative twist. When we carry out the twisting operation, we add or subtract $1$ from the tangle number, depending on the direction of twist.

\hspace{-3 mm}
\begin{minipage}{0.5\textwidth}
\begin{figure}[H]
\centering
\vspace{5 mm}
\scalebox{0.5}{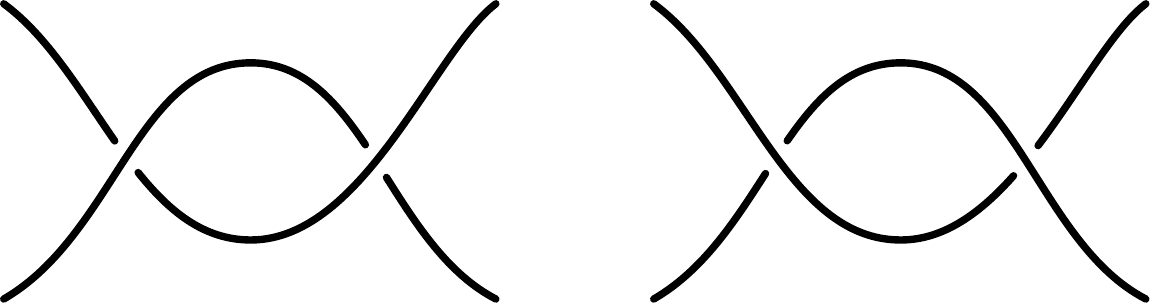}
\vspace{5 mm}
\caption{Tangles with number $2$ and $-2$}
\end{figure}
\end{minipage}
\hspace{-2 mm}
\begin{minipage}{0.5\textwidth}
\begin{figure}[H]
\centering
\scalebox{0.5}{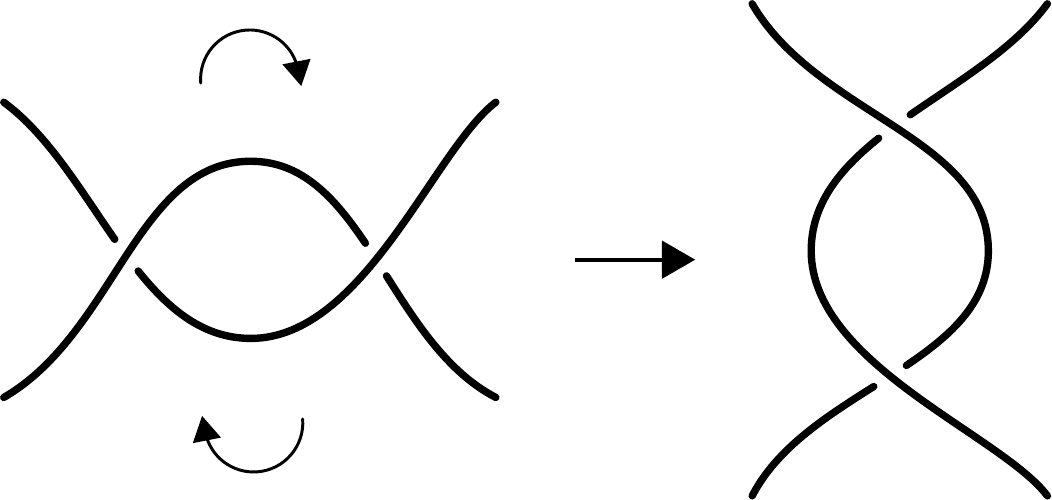}
\caption{Rotation}
\label{rotation}
\end{figure}
\end{minipage}

\vspace{5 mm}
Secondly, we have the rotation operation. As the name suggests, we rotate the whole tangle $90^\circ$ clockwise for this operation as in Figure \ref{rotation} above. When we carry out this operation, the tangle number is transformed to its negative reciprocal. Such tangles that are constructed using only these two operations are called rational tangles.

\begin{myexa}
We begin with the untangle. We twist it three times in the negative direction, rotate it and do one twist in the negative direction. Finally, we rotate it and do two more twists in the positive direction. The resulting tangle is given in Figure \ref{-312} below. By following the rules stated above, the tangle number of this tangle can be calculated to be $\frac{7}{2}$.

\begin{figure}[H]
\centering
\scalebox{0.5}{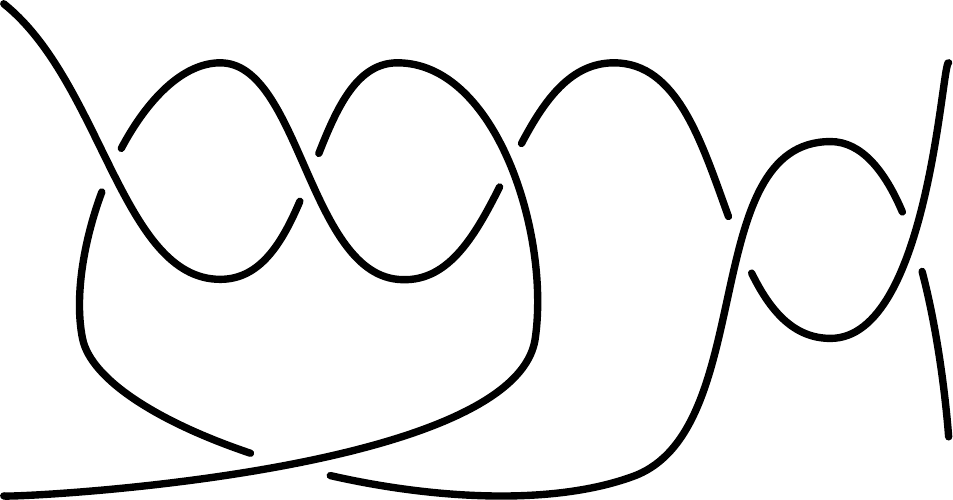}
\caption{A tangle with number $\frac{7}{2}$}
\label{-312}
\end{figure}

\end{myexa}

What is the significance of tangle numbers? A remarkable theorem by John Conway is given below:

\begin{mythm} [Conway's Theorem]
Two rational tangles with equal tangle numbers are equivalent (i.e. can be transformed from one to the other via a sequence of Reidemeister moves with the four ends fixed).
\end{mythm}

\begin{myrem} If you are not familiar with knot theory, Reidemeister moves is just a set of moves which are allowed locally in a knot. They are useful in defining knot invariants, which is a heavily studied subject in knot theory. It is one of the basic topics defined early on in any knot theory course or textbooks.
\end{myrem}

For the proof of Conway's theorem, refer to the paper by  J.R. Goldman and L.H. Kauffman \cite{goldmankauffman}. By the theorem, if we can somehow get the number down to $0$, the resulting tangle will be equivalent to the untangle, thus outlining a systematic method to untangle it.

Amazingly, we can find a way to untangle a rational tangle using the Euclidean algorithm. If we are given a tangle with number $\frac{x_0}{x_1} \geq1$, we can find a way to untangle it by finding the Euclidean algorithm for the pair $x_0$ and $x_1$ and by manipulating the equations slightly (dividing through the $i$-th equation with $x_i$ and multiplying through some equations with $-1$ accordingly). However, by Corollary \ref{regularalgmin} we note that the regular Euclidean algorithm gives the least number of steps among any Euclidean algorithms, so by fixing each $\epsilon_i$ in (\ref{euclidany}) to be $+1$ and multiplying alternate equations with $-1$, we have:
\begin{eqnarray} \label{untangleeuclid}
\frac{x_0}{x_1}-p_1&=& \frac{x_2}{x_1} \nonumber
\\ -\frac{x_1}{x_2}+p_2&=& -\frac{x_3}{x_2} \nonumber
\\  & \vdots
\\ (-1)^{m-2}\frac{x_{m-2}}{x_{m-1}}-(-1)^{m-2}p_{m-1}&=&(-1)^{m-2}\frac{x_{m}}{x_{m-1}} \nonumber
\\ (-1)^{m-1}\frac{x_{m-1}}{x_m}- (-1)^{m-1}p_m &=&0. \nonumber
\end{eqnarray}

Why did we multiply every alternate equation with $-1$? Note that when we carry out a rotation, we take the negative reciprocal of the tangle number. The gist for the algorithm in (\ref{untangleeuclid}) is that we begin with a tangle with number $\frac{x_0}{x_1}$ and add $-p_1$ twist to it, to get a tangle with number $\frac{x_2}{x_1}$. Then, the next step is to rotate this tangle we get a tangle with number $-\frac{x_1}{x_2}$. But in the regular Euclidean algorithm, in the following step, we have the equation for $\frac{x_1}{x_2}$, not the negative. Thus, we need to multiply this equation with $-1$ to ensure continuity in untangling algorithm. Going down the list of equations, we see that if we have the regular Euclidean algorithm, we need to multiply every alternate equation with $-1$, so that we have a coherent untangling algorithm related to the list of equations.

From this, we can read off the algorithm to untangle the $\frac{x_0}{x_1}$ tangle: we first twist it $p_1$ times in the negative direction, rotate it, twist another $p_2$ times, rotate, and so on, going down the algorithm and eventually getting to $0$, untangling it. Therefore, there is a total of $(m-1)+\sum_{i=1}^mp_i$ steps to untangle the said tangle. 

A similar method can be done for tangles with numbers $\frac{x_0}{x_1}\leq -1$. For the tangles with numbers $-1<\frac{x_0}{x_1}<1$, we start with a rotation to get the magnitude of the number to be greater than $1$ and then continue accordingly. 

Thus, any rational tangle can be untangled in a finite number of moves by looking at the corresponding Euclidean algorithm, proving the existence of an untangling algorithm.

Furthermore, note that by doing the algebraic manipulations on the Euclidean algorithm, they do not change the number of steps required in the algorithm as neither the value of the $p_i$'s nor the number of equations are affected. Thus, we can determine the minimum number of steps required to untangle a rational tangle with number $\frac{x_0}{x_1}$ by finding the number of steps to carry out the regular Euclidean algorithm for the pair $|x_0|$ and $|x_1|$. 

\subsection{Untangling Algorithm with Minimal Permutations of Ends}

Using any Euclidean algorithm, we can determine the steps to untangle any given rational tangle. We look at a special method discussed in the bulk of the previous sections: the method of least absolute remainders. 

By a similar method as the regular Euclidean algorithm above, we can also show that the method of least absolute remainders can also be used to determine the untangling algorithm of a rational tangle. However, we must be careful because we only need to multiply some equations with $(-1)$, rather than every alternate equations. This is because in the regular Euclidean algorithm, all of our remainders are positive but in the method of least absolute remainders, we may have negative remainders in some of the equations.

Recall that we have shown that the method of least absolute remainders also gives the least number of steps in the algorithm. Furthermore, this method gives the least number of equations, by Corollary \ref{regularalgmin2}. 

Translating into the language of tangles, the method of least absolute remainders gives the least number of steps for untangling rational tangles and this method has the least number of rotations involved. Twisting permutes two of the ends of the tangle, but rotation permutes all four of the ends. So ideally, if we want to untangle a rational tangle with the least number of permutations of the ends, the method of least absolute remainders is the way to go as we have the same number of steps as the regular Euclidean algorithm, but with less number of rotations involved.

\subsection{Restricted Untangling Algorithm}
What about the method of negative remainders? If we consider the method of negative remainders in (\ref{untangleeuclid}), before we multiply any equations with $-1$, we note that all the terms in the RHS are negative, so when we rotate these tangles, they will give positive fractions. Hence, we do not need to multiply any equations with $-1$. By doing so, we note that all the twists in the algorithm are negative twists. Therefore, if we require an algorithm to untangle a rational tangle such that the twists are restricted only to one direction, we utilise the method of negative remainders.

\begin{myexa}
Consider the tangle with number $\frac{8}{5}$ as in Figure \ref{85} below. We want to find out the ways to untangle the tangle using the three different types of Euclidean algorithms.

\begin{figure}[H]
\centering
\scalebox{0.55}{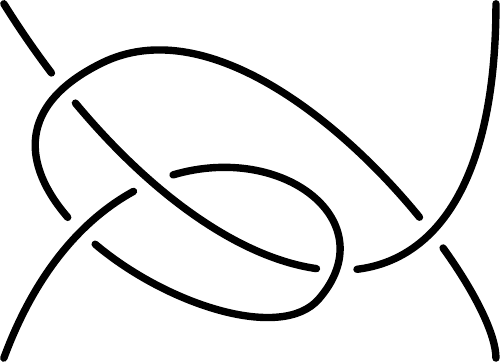}
\caption{A tangle with number $\frac{8}{5}$}
\label{85}
\end{figure}

We begin by writing down the three Euclidean algorithm for the numbers $8$ and $5$ below. The first column is for the regular Euclidean algorithm, the second is the method of least absolute remainders and the third one is the method of negative remainders.

\begin{minipage}{0.3\textwidth}
\begin{eqnarray*} 
8&=& 5(1)+3
\\ 5&=& 3(1)+2
\\ 3&=& 2(1)+1
\\ 2&=& 1(2)+0
\end{eqnarray*}
\end{minipage} 
\begin{minipage}{0.3\textwidth}
\begin{eqnarray*} 
8&=& 5(2)-2
\\ 5&=& 2(2)+1
\\ 2&=& 1(2)+0
\\
\end{eqnarray*}
\end{minipage} 
\begin{minipage}{0.3\textwidth}
\begin{eqnarray*} 
8&=& 5(2)-2
\\ 5&=& 2(3)-1
\\ 2&=& 2(1)+0
\\
\end{eqnarray*}
\end{minipage}

\vspace{5 mm}
By the manipulations discussed earlier, we would get the list of equations below, respective to the Euclidean algorithm used above.

\begin{minipage}{0.3\textwidth}
\begin{eqnarray*} 
\frac{8}{5}-1&=&\frac{3}{5}
\\ -\frac{5}{3}+1&=&-\frac{2}{3}
\\ \frac{3}{2}-1&=&\frac{1}{2}
\\ -2+2&=&0
\end{eqnarray*}
\end{minipage} 
\begin{minipage}{0.3\textwidth}
\begin{eqnarray*} 
\frac{8}{5}-2&=& -\frac{2}{5}
\\ \frac{5}{2}-2&=& \frac{1}{2}
\\- 2+2&=&0
\\
\end{eqnarray*}
\end{minipage} 
\begin{minipage}{0.3\textwidth}
\begin{eqnarray*} 
\frac{8}{5}-2&=& -\frac{2}{5}
\\ \frac{5}{2}-3&=& -\frac{1}{2}
\\ 2-2&=& 0
\\
\end{eqnarray*}
\end{minipage}

\vspace{5 mm}
From these, we can read off the algorithm for untangling the rational tangle using the three different Euclidean algorithm. Denoting $T$ as positive twist, $-T$ as negative twist and $R$ as rotation, we would get three different ways of untangling the $\frac{8}{5}$ tangle:

\begin{quote}
Regular Euclidean algorithm: $[-T,R,T,R,-T,R,T,T].$
\\Method of least absolute remainders: $[-T,-T, R,-T,-T,R,T,T].$
\\Method of negative remainders: $[-T,-T,R,-T,-T,-T,R,-T,-T].$
\end{quote}

Note that the regular Euclidean algorithm and the method of least absolute remainders gives $8$ number of steps, and by the theorems proved earlier, this is the minimum possible number of steps to untangle the given tangle. Furthermore, in the regular Euclidean algorithm, we have $3$ rotations whereas in the method of least absolute remainders, there are only $2$ rotations involved. Thus, in terms of number of permutation of the ends of the tangles, the method of least absolute remainders is more efficient as there are less rotations required in the algorithm.

For the method of negative remainders, the only twist moves involved are the negative twists. Therefore, if our twists are restricted to only one direction, we can use the method of negative remainders to find the untangling algorithm. 

\end{myexa}

\section*{Acknowledgement}
I would like to thank Professor David Gauld of University of Auckland for accepting me as a summer research student under the Midyear Research Scholarship. His guidance and direction was very crucial in the research and writing of this paper. It has been an enjoyable experience working under his supervision.


\begin{thebibliography}{9}

\bibitem{burton}
Burton, D.M. (1980).
\emph{Elementary Number Theory}.
Allyn and Bacon, Boston.

\bibitem{cromwell}
Cromwell, P.R. (1980).
\emph{Knots and Links}.
Cambridge University Press, Cambridge.


\bibitem{goldmankauffman}
Goldman, J.R. and Kauffman, L.H. (1996).
\emph{Rational Tangles, Advances in Applied Mathematics}: 300-332. 

\bibitem{goodman}
Goodman, A.W. and Zaring, W.M. (1952).
\emph{Euclid's Algorithm and the Least-Remainder Algorithm, The American Mathematical Monthly}: 156-159. 

\bibitem{uspensky}
Uspensky J.V. \& Heaslet M.A. (1939).
  \emph{Elementary Number Theory}, 1st Edition. McGraw-Hill Book Company, New York.




\end{thebibliography}
\end{document}